\documentclass[12pt]{article}

\usepackage{amssymb,amsmath,latexsym,amsthm}
\usepackage[english]{babel}
\usepackage[all]{xy}

\newtheorem{theorem}{Theorem}[section]
\newtheorem{lemma}{Lemma}[section]
\newtheorem{proposition}{Proposition}[section]
\newtheorem{cor}{Corollary}[section]

\theoremstyle{definition}

\newtheorem*{example}{Example}

\begin{document}

\title{Spinor class field for generalized Eichler orders}


\author{\sc Luis Arenas-Carmona}


\newcommand\Q{\mathbb Q}
\newcommand\alge{\mathfrak{A}}
\newcommand\Da{\mathfrak{D}}
\newcommand\Ea{\mathfrak{E}}
\newcommand\Ha{\mathfrak{H}}
\newcommand\oink{\mathcal O}
\newcommand\matrici{\mathbb{M}}
\newcommand\Txi{\lceil}
\newcommand\ad{\mathbb{A}}
\newcommand\enteri{\mathbb Z}
\newcommand\finitum{\mathbb{F}}
\newcommand\bbmatrix[4]{\left(\begin{array}{cc}#1&#2\\#3&#4\end{array}\right)}

\maketitle

\begin{abstract}
We compute the spinor class field for a genus of orders, in a central simple algebra of higher dimension,
 that are intersections of two maximal orders. 
In particular, we compute the number of spinor genera in a genus of such orders, as the degree of an explicit 
extension of class fields. 

\end{abstract}

\bigskip
\section{Introduction}

Let $K$ be a number field. Let $\alge$ be a central simple
$K$-algebra ($K$-CSA or CSA over $K$)
 of dimension $n^2>4$. By a generalized Eichler order (or GEO) we mean the intersection of two
maximal orders.  Two orders in $\alge$ are said to be the same genus if and only if they are locally conjugate
at all places of $K$. For any genus  $\mathbb{O}$ of orders of maximal rank we can define its
 spinor class field $\Sigma/K$, an abelian extension that classifies conjugacy classes in $\mathbb{O}$.
 More precisely, for any order $\Da\in\mathbb{O}$, the field
 $\Sigma$ is the class field corresponding to the class group $K^*H(\Da)$,
where $H(\Da)$ is the group of reduced norms of elements in the adelization $\alge_\ad^*$ of
$\alge^*$ stabilizing $\Da$ by conjugation. In \cite[\S3]{spinor} we proved the existence of an explicit map
$$\rho:\mathbb{O}\times\mathbb{O}\rightarrow\mathrm{Gal}(\Sigma/K),$$
with the following properties:
\begin{enumerate}
\item $\Da$ and $\Da'$ are in the same  conjugacy class if and only if
$\rho(\Da,\Da')=\mathrm{Id}_\Sigma$, \item
$\rho(\Da,\Da'')=\rho(\Da,\Da')\rho(\Da',\Da'')\qquad\forall
(\Da,\Da',\Da'')\in\mathbb{O}^3$.
\end{enumerate}
  This holds also for quaternion $K$-algebras satisfying Eichler's condition (\S2). In fact,
the definition of spinor class field extends easily to general quaternion algebras or CSAs over global functions fields
via the theory of spinor genera \cite{abelianos}. The class group defining $\Sigma$, for any genus $\mathbb{O}$
of Eichler orders, is already implicit in \cite[Cor. III.5.7]{Vigneras}. However, for CSAs of dimension $3^2$ or larger, only
the spinor class field for maximal orders was previously  known explicitly
\cite{spinor}:
\begin{quote}
Let $\alge$ be an $n^2$-dimensional $K$-CSA.
For every place $\wp$ of $K$,
let $f_\wp(\alge/K)=f$, where $\alge_\wp\cong\matrici_f(E_\wp)$, for a local division algebra
$E_\wp$ over $K_\wp$. Then
the spinor class field of maximal orders in $\alge$ is the maximal exponent-$n$
sub-extension $\Sigma_0$ of the wide Hilbert class field of $K$ satisfying the following conditions:
\begin{enumerate}\item $f_\wp(\Sigma_0/K)$ divides $f_\wp(\alge/K)$ at all finite places. \item  $\Sigma_0/K$ splits at every real place $\wp$ of $K$ where $f_\wp(\alge/K)=n$.
\end{enumerate} 
\end{quote}
In the same language, the corresponding result, for Eichler orders in the quaternionic case, is as follows \cite{Eichler2}:
\begin{quote}
The spinor class field for Eichler orders of level $I=\prod_\wp\wp^{\alpha(\wp)}$ in a quaternion algebra $\alge$
 over $K$ is the maximal exponent-$2$
sub-extension $\Sigma$, of the spinor class field $\Sigma_0$ of maximal orders, splitting at all places where $\alpha(\wp)$ is odd.
\end{quote}

The purpose of the current work is to give a similar result for GEOs on a $K$-CSA $\alge$
of arbitrary dimension. To state this result we need some facts about local GEOs.

\subparagraph{Locally symmetric GEOs.}
Let $k=K_\wp$ be a local field and let $\alge=\matrici_f(B)$ be a $k$-CSA, where $B$ is a division algebra.
 Recall that $B^f$, the space of column vectors, is naturally
a left $\matrici_f(B)$-module and a right $B$-module, and this bi-module structure is the one considered
throughout this paper. Every maximal order in $\alge$ has the form $\Da_\Lambda=\{a\in\alge|a\Lambda\subseteq\Lambda\}$,
for some Lattice $\Lambda$ on the column space $B^f$ satisfying $\Lambda\oink_B=\Lambda$, 
where $\oink_B$ is the  maximal order of $B$. Such lattices are called $\oink_B$-lattices.
Note that, for $\lambda\in B^*$, the map $x\mapsto x\lambda$ is not a $B$-module homomorphism
unless $\lambda$ is central,
but $\Lambda\mapsto \Lambda\lambda$ define an action of $B^*$ on the set of lattices since
$\lambda\oink_B=\oink_B\lambda$. In these notations,
 $\Da_\Lambda=\Da_M$ if and only if $M=\Lambda\lambda$ for some $\lambda\in B^*$.

Let $\Lambda$ and $M$ be two $\oink_B$-lattices in $B^f$ and  let $\pi$ be a uniformizing parameter of $B$.
By the theory of invariant factors, there exists a $B$-basis $\{e_1,\dots,e_f\}$ of $B^f$, such that
$$\Lambda=e_1\oink_B+e_2\oink_B+\cdots+e_f\oink_B, \quad
M=\pi^{r_1}e_1\oink_B+\pi^{r_2}e_2\oink_B+\cdots+\pi^{r_f}e_f\oink_B,$$
where $r_1\leq r_2\leq\cdots\leq r_f$. The integers  $r_1,\cdots,r_f$ are call the invariant exponents of $(\Lambda,M)$.
The class $\rho_\wp(\Da_\Lambda,\Da_M)=\overline{r_1+\dots+r_f}\in\enteri/f\enteri$ is called the total (local) distance between the maximal orders 
$\Da_\Lambda$ and $\Da_M$. This distance is a conjugation invariant of $\Da$. The vector
$( r_1,r_2,\dots,r_f)\in\enteri^n$ is called the type-distance from $\Lambda$ to $M$ and its image in 
$\enteri^n/\langle(1,1,\dots,1)\rangle$ is an invariant of the pair $(\Da_\Lambda,\Da_M)$.
 A local GEO $\Da=\Da_\Lambda\cap\Da_M$ is said to be symmetric if and only
if the type distance from $\Lambda$ to $M$ satisfies
$$(r_2-r_1, r_3-r_2,\dots,r_f-r_{f-1})=(r_f-r_{f-1},r_{f-1}-r_{f-2},\dots,r_2-r_1).$$
A global GEO is locally symmetric if every completion is symmetric. 
Two global GEOs $\Da=\Da_1\cap\Da_2$ and $\Da'=\Da'_1\cap\Da'_2$ 
 are in the same genus if and only if the type distances of the pairs 
$(\Da_1,\Da_2)$ and $(\Da'_1,\Da'_2)$ coincide at every finite place,
at follows easily from the definition of invariant factors. 
 
\begin{theorem}\label{t11}
The spinor class field for a global GEO $\Da=\Da_1\cap\Da_2$ in a $K$-CSA $\alge$
is the maximal  subextension $\Sigma$, of the spinor class field $\Sigma_0$ for
maximal orders, whose local inertia degree
$f_\wp(\Sigma/K)$ divides  the total distance $\rho_\wp(\Da_1,\Da_2)$ 
 at every place $\wp$ where $\Da_\wp$ is symmetric.
\end{theorem}
Note that the symmetry condition can be rewritten $r_m-r_{m-1}=r_{f-m+2}-r_{f-m+1}$, so that $r_m+r_{f-m+1}$
is independent of $m$. We conclude that $2(r_1+\dots+r_f)$ is always a multiple of $f=f_\wp(\alge/K)$. We conclude
that $\Sigma_0/\Sigma$ is always an exponent-2 extension.

\begin{cor}
The spinor class field for a GEO  in an odd dimensional CSA $\alge$
 over $K$ is the spinor class field $\Sigma_0$  of maximal orders. In particular, the number of
conjugacy classes in a genus of GEOs on an odd-dimensional CSA is independent of the genus.
\end{cor}

\section{The theory of spinor class field}

In this section we review the basic facts about spinor genera and spinor class fields of orders.
See \cite{spinor} for details. In all that follows $\alge$ is a CSA over a global field $K$ and $S$ a non-empty finite 
set of places of $K$ containing the archimedean  places if any.  Let $\Da$ be an $S$-order of maximal
rank in $\alge$, or as we say henceforth, an order on $\alge$. Let $\Ha$ be an arbitrary $S$-suborder of $\Da$. Completions at a place $\wp\in U=S^c$ are denoted $\Da_\wp$ or $\Ha_\wp$. Let $\Pi(K)$ be the set of all places in $K$, let $\ad\subseteq\prod_{\wp\in \Pi(K)}K_\wp$ be the adele ring of $K$, and let $\alge_\ad=\alge\otimes_K\ad$ be the adelization of the algebra. If $a=(a_\wp)_\wp\in\alge_\ad$ is an adelic element, we let $a\Da a^{-1}$ denote the order $ \Da'$ defined locally by $\Da'_\wp=a_\wp\Da_\wp a_\wp^{-1}$ at finite places.

Since any two maximal orders are locally conjugate at all places, if we fix a maximal order $\Da$, any other maximal order on $\alge$ has the form $\Da'=a\Da a^{-1}$ for some adelic element $a\in\alge_\ad^*$. More generally, it is said that
two orders $\Da$ and $\Da'$, on $\alge$,  are in the same genus if $\Da'=a\Da a^{-1}$ for some adelic element $a$.
 We say that $\Da$ and $\Da'$ are in the same spinor genus if $a$ can be chosen of the form $a=bc$ where $b\in\alge$ and $N(c)=1_\ad$, where $N:\alge_\ad^*\rightarrow \ad^*=:J_K$ is the reduced norm on adeles. We write $\Da'\in\mathrm{Spin}(\Da)$. The spinor class field $\Sigma=\Sigma(\Da)$ is defined as the class field corresponding to the group $K^*H(\Da)\subseteq J_K$, where
$$H(\Da)=\{N(a)|a\in\alge_\ad^*,\  a\Da a^{-1}=\Da\}.$$
Let $t\mapsto [t,\Sigma/K]$ denote the Artin map on ideles.
The distance between the orders $\Da$ and $\Da'\in\mathrm{gen}(\Da)$ is the element
$\rho(\Da,\Da')\in\mathrm{Gal}(\Sigma/K)$ defined by $\rho(\Da,\Da')=[N(a),\Sigma/K]$, for any adelic element $a\in\alge_\ad^*$
satisfying $\Da'=a\Da a^{-1}$. Note that this implies $\rho(\Da,\Da'')=\rho(\Da,\Da')\rho(\Da',\Da'')$ for any triple $(\Da,\Da',\Da'')$ of orders in a genus $\mathbb{O}$. Two orders are in the same spinor genus if and only if their distance is trivial.
One important property of spinor genera is that they coincide with conjugacy classes whenever strong approximation holds,
so in this case two orders are conjugate if and only if their distance is trivial.
 In the present setting, strong approximation is equivalent to the following property:
\begin{quote} \textbf{Generalized Eichler Condition (GEC):} There exists a place $\wp\in S$ such that the local
 algebra $\alge_\wp$ is not a division algebra.
\end{quote}
In the particular case where $K$ is a number field, and $S$ is the set of archimedean places, GEC reduces to the better known
Eichler Condition:
\begin{quote} \textbf{Eichler Condition (EC):} Either $n>2$ or $\alge$ is unramified  at some archimedean place.
\end{quote}

Note for future reference that $H(\Da)=J_K\cap\prod_{\wp\in\Pi(K)}H_\wp(\Da)$, where $$H_\wp(\Da)=\{N(a)|a\in\alge_\wp^*,\  a\Da_\wp a^{-1}=\Da_\wp\}.$$ The sets $H(\Da)$ and $H_\wp(\Da)$ are called global and local spinor image, respectively. When $k$ is an arbitrary local field, we also write $H_k(\Ea)$ for the spinor image of a local order
$\Ea$, which is defined analogously. 

\section{blocks in Weil apartments.}
 
In all of this section, let $k$ be a local field, and let $B$ be a central division $k$-algebra with uniformizing parameter $\pi$. 
Let $\mathfrak{B}$ be the Bruhat-Tits building (or BT-building) associated to $\mathrm{PGL}_n(B)$ as defined in \cite{build}.
Recall that the vertices of $\mathfrak{B}$ are in one to one correspondence with the maximal orders in 
$\matrici_n(B)$. An apartment is the maximal sub-complex whose vertices correspond to maximal orders containing a fixed conjugate of the order $\mathfrak{P}=\bigoplus_{i=1}^n\oink_B E_{i,i}$ of integral diagonal matrices, where $\{E_{i,j}\}_{i,j}$ is the canonical $B$-basis of $\matrici_n(B)$. 
Consider the apartment $A_0$ corresponding to $\mathfrak{P}$, which we call the standard apartment.
Note that the set of maximal orders in $A_0$  is in correspondence with the homothety classes of left fractional $\mathfrak{P}$-ideals in $k\mathfrak{P}$.
 In other words they are the stabilizers $\Da_{(a_1,\dots,a_n)}$ of the lattices of the form 
$\pi^{a_1}e_i\oink_B+\cdots+\pi^{a_n}e_n\oink_B$, where $\{e_1,\dots,e_n\}$ is the cannonical basis of
the column space $B^f$, and $\stackrel\rightarrow{a}=(a_1,\dots,a_n)\in\enteri^n$.
Let $\stackrel{\rightarrow}u=(1,\dots,1)\in\enteri^n$. Since it is clear that
$\Da_{\stackrel{\rightarrow}a+m\stackrel{\rightarrow}u}=\Da_{\stackrel{\rightarrow}a}$ 
for any $\stackrel{\rightarrow}a\in\enteri^n$ and any $m\in\enteri$, we also use the notation
$$\Da_{(a_1,\dots,a_n)}=:\Da_{[a_2-a_1,\dots,a_n-a_{n-1}]}.$$
Note that the sub-index on the right can be seen as an element in the quotient group
$\Gamma=\enteri^n/\langle\stackrel{\rightarrow}u\rangle$. Elements of $\Gamma$ are denoted
in brackets, e.g., $[b]$ and $[d]$, in all that follows. Furthermore, if $[b]=[b_1,\dots,b_{n-1}]$
and $\stackrel{\rightarrow}a=(a_1,\dots,a_n)$, we write
$[b]= \stackrel{\rightarrow}a+\langle\stackrel{\rightarrow}u\rangle$, if $b_i=a_{i+1}-a_i$ for every
$i=1,\dots,n-1$. On $\Gamma$ we define the total length
function $$\Big|\Big|[b_1,b_2,\dots,b_n]\Big|\Big|=|b_1|+|b_2|+\dots+|b_n|.$$ 
Furthermore, the permutation group $S_n$ acts naturally on the order
$\mathfrak{P}$ and its generated $K$-algebra.
This define a natural action of $S_n$ on the group of fractional ideals of $\mathfrak{P}$ that can be
interpreted as an action on the vertices of the chamber.
\begin{example} If $n=5$, the permutation $\sigma=(12)(345)$ satisfies $$\sigma(\Da_{[1,2,3,4]})=
\sigma(\Da_{(0,1,3,6,10)})=\Da_{(1,0,10,3,6)}=\Da_{[-1,10,-7,3]}.$$
\end{example}
\begin{example}
If $n=3$, the orbit of $\Da_{[2,1]}$ is the set
 $$\left\{\Da_{[2,1]},\Da_{[3,-1]},\Da_{[-2,3]},\Da_{[1,-3]},
\Da_{[-3,2]},\Da_{[-1,-2]}\right\}.$$
\end{example}
Next result is immediate from the definition:

\begin{lemma}
$\Da_0$ is the only vertex in the standard apartment stabilized by the whole of $S_n$.
Every $S_n$-orbit contains a unique order of the form $\Da_{[b_1,\dots,b_{n-1}]}$ with
$b_1,\dots,b_{n-1}\geq0$.
\end{lemma}

We call either, an element $[b_1,\dots,b_{n-1}]$ with $b_1,\dots,b_{n-1}\geq0$, or the corresponding order $\Da_{[b_1,\dots,b_{n-1}]}$, totally positive. Note that $[b]= \stackrel{\rightarrow}a+\langle\stackrel{\rightarrow}u\rangle$
is totally positive if and only if $\stackrel{\rightarrow}a=(a_1,\dots,a_n)$ is an increasing sequence.

\begin{lemma}
Assume $[b]$ is totally positive. 
Then the maximal orders containing the generalized Eichler
order $\Da=\Da_0\cap\Da_{[b]}$ are exactly the orders 
$\Da_{[c]}$ with $[c]=[c_1,\dots,c_{n-1}]$ and
$0\leq c_i\leq b_i$ for every $i=1,\dots,n-1$.
\end{lemma}

\begin{proof}
First note that both $\Da_0$ and $\Da_{[b]}$ contain the order $\mathfrak{P}$
of integral diagonal matrices, so the same hold for every order containing their intersection. 
We conclude that every such maximal order is in the apartment defined by $\mathfrak{P}$.
Next, choose $a_1,\dots,a_n$ and $d_1,\dots,d_n$, in a way that $a_{i+1}=a_i+b_i$ and $d_{i+1}=d_i+c_i$
for every $i=1,\dots,n-1$. Then $a_j-a_i\geq d_j-d_i$ for every pair $(i,j)$ with $1\leq i<j\leq n$. The
result follows if we observe that $$\Da_{[b]}=\left(\begin{array}{ccccc}
\oink_K&\pi^{a_1-a_2}\oink_K&\pi^{a_1-a_3}\oink_K&\cdots&\pi^{a_1-a_n}\oink_K\\
\pi^{a_2-a_1}\oink_K&\oink_K&\pi^{a_2-a_3}\oink_K&\cdots&\pi^{a_2-a_n}\oink_K\\
\pi^{a_3-a_1}\oink_K&\pi^{a_3-a_2}\oink_K&\oink_K&\cdots&\pi^{a_3-a_n}\oink_K\\
\vdots&\vdots&\vdots&\ddots&\vdots\\
\pi^{a_n-a_1}\oink_K&\pi^{a_n-a_2}\oink_K&\pi^{a_n-a_3}\oink_K&\cdots&\oink_K\end{array}\right),$$
and a similar formula holds for every order in the apartment.
\end{proof}

In what follows we denote by $\mathfrak{S}_0(\Ha)$, for every order $\Ha$, the maximal sub-complex $\mathfrak{S}$ of the BT-building $\mathfrak{B}$ such that every vertex of $\mathfrak{S}$ corresponds to a maximal order containing $\Ha$,
and call it the block of $\Ha$. 
Note that if $\Ha'$ is the intersection of all maximal orders containing $\Ha$,
then $\mathfrak{S}_0(\Ha)=\mathfrak{S}_0(\Ha')$. We let $S_0(\Ha)$ denote the set of vertices of $\mathfrak{S}_0(\Ha)$.

\begin{example}
The cell-complexes $\mathfrak{S}_0(\Da')$, for $\Da'$ in the $S_3$-orbit of $\Da=\Da_0\cap\Da_{[2,1]}$, are 
as described in Table 2.
\begin{table}
$$
\unitlength 1mm 
\linethickness{0.4pt}
\ifx\plotpoint\undefined\newsavebox{\plotpoint}\fi 
\begin{tabular}{|c|c|c|c|} \hline
$\sigma$&$\mathrm{Id}$&$(2\ 3)$&$(1\ 2)$\\ \hline
\begin{picture}(20,26)(0,0)
\put(9,13){\makebox(0,0)[cc]{$S_0\Big(\sigma(\Da)\Big)$}}
\end{picture}&
\begin{picture}(26,26)(0,0)
\put(12,0){\vector(0,3){25}} \put(0,12){\vector(3,0){25}}
\put(12.2,12){\circle*{.7}} \put(12.2,16){\circle*{.7}}
\put(16.2,12){\circle*{.7}} \put(16.2,16){\circle*{.7}}
\put(20.2,12){\circle*{.7}} \put(20.2,16){\circle*{.7}}
\put(16.2,16){\line(0,-1)4}\put(12.2,16){\line(1,0)4}
\put(16.2,16){\line(1,-1)4}\put(12.2,16){\line(1,-1)4}
\put(20.2,16){\line(0,-1)4}\put(16.2,16){\line(1,0)4}
\put(15.3,15){\makebox(0,0)[cc]{\tiny{$b$}}}
\put(13.3,13){\makebox(0,0)[cc]{\tiny{$a$}}}
\put(19.3,15){\makebox(0,0)[cc]{\tiny{$d$}}}
\put(17.3,13){\makebox(0,0)[cc]{\tiny{$c$}}}
\end{picture}&
\begin{picture}(26,26)(0,0)
\put(12,0){\vector(0,3){25}} \put(0,12){\vector(3,0){25}}
\put(12.2,12){\circle*{.7}} \put(16.2,8){\circle*{.7}}
\put(16.2,12){\circle*{.7}} \put(20.2,8){\circle*{.7}}
\put(20.2,12){\circle*{.7}} \put(24.2,8){\circle*{.7}}
\put(16.2,12){\line(0,-1)4}\put(20.2,12){\line(0,-1)4}
\put(16.2,8){\line(1,0)4}\put(20.2,8){\line(1,0)4}
\put(12.2,12){\line(1,-1)4}\put(16.2,12){\line(1,-1)4}\put(20.2,12){\line(1,-1)4}
\put(15.3,11){\makebox(0,0)[cc]{\tiny{$a$}}}
\put(17.3,9){\makebox(0,0)[cc]{\tiny{$b$}}}
\put(19.3,11){\makebox(0,0)[cc]{\tiny{$c$}}}
\put(21.3,9){\makebox(0,0)[cc]{\tiny{$d$}}}
\end{picture}&
\begin{picture}(26,26)(0,0)
\put(12,0){\vector(0,3){25}} \put(0,12){\vector(3,0){25}}
\put(12.2,12){\circle*{.7}} \put(12.2,16){\circle*{.7}}
\put(8.2,16){\circle*{.7}} \put(8.2,20){\circle*{.7}}
\put(4.2,20){\circle*{.7}} \put(4.2,24){\circle*{.7}}
\put(8.2,20){\line(0,-1)4}\put(4.2,24){\line(0,-1)4}
\put(8.2,16){\line(1,0)4}\put(4.2,20){\line(1,0)4}
\put(8.2,20){\line(1,-1)4}\put(4.2,24){\line(1,-1)4}\put(8.2,16){\line(1,-1)4}\put(4.2,20){\line(1,-1)4}
\put(11.3,15){\makebox(0,0)[cc]{\tiny{$a$}}}
\put(9.3,17){\makebox(0,0)[cc]{\tiny{$b$}}}
\put(7.3,19){\makebox(0,0)[cc]{\tiny{$c$}}}
\put(5.3,21){\makebox(0,0)[cc]{\tiny{$d$}}}
\end{picture}\\  \hline
\end{tabular}
$$
$$
\unitlength 1mm 
\linethickness{0.4pt}
\ifx\plotpoint\undefined\newsavebox{\plotpoint}\fi 
\begin{tabular}{|c|c|c|c|} \hline
$\sigma$&$(1\ 3\ 2)$&$(1\ 2\ 3)$&$(1\ 3)$\\ \hline
\begin{picture}(20,26)(0,0)
\put(9,13){\makebox(0,0)[cc]{$S_0\Big(\sigma(\Da)\Big)$}}
\end{picture}&
\begin{picture}(26,26)(0,0)
\put(12,0){\vector(0,3){25}} \put(0,12){\vector(3,0){25}}
\put(12.2,12){\circle*{.7}} \put(12.2,8){\circle*{.7}}
\put(12.2,4){\circle*{.7}} \put(16.2,8){\circle*{.7}}
\put(16.2,4){\circle*{.7}} \put(16.2,0){\circle*{.7}}
\put(16.2,8){\line(0,-1)4}\put(16.2,4){\line(0,-1)4}
\put(12.2,8){\line(1,0)4}\put(12.2,4){\line(1,0)4}
\put(12.2,12){\line(1,-1)4}\put(12.2,8){\line(1,-1)4}\put(12.2,4){\line(1,-1)4}
\put(13.3,9){\makebox(0,0)[cc]{\tiny{$a$}}}
\put(15.3,7){\makebox(0,0)[cc]{\tiny{$b$}}}
\put(13.3,5){\makebox(0,0)[cc]{\tiny{$c$}}}
\put(15.3,3){\makebox(0,0)[cc]{\tiny{$d$}}}
\end{picture}&
\begin{picture}(26,26)(0,0)
\put(12,0){\vector(0,3){25}} \put(0,12){\vector(3,0){25}}
\put(12.2,12){\circle*{.7}} \put(8.2,12){\circle*{.7}}
\put(8.2,16){\circle*{.7}} \put(4.2,16){\circle*{.7}}
\put(4.2,20){\circle*{.7}} \put(0.2,20){\circle*{.7}}
\put(8.2,16){\line(0,-1)4}\put(4.2,20){\line(0,-1)4}
\put(0.2,20){\line(1,0)4}\put(4.2,16){\line(1,0)4}
\put(4.2,16){\line(1,-1)4}\put(8.2,16){\line(1,-1)4}\put(0.2,20){\line(1,-1)4}\put(4.2,20){\line(1,-1)4}
\put(9.3,13){\makebox(0,0)[cc]{\tiny{$a$}}}
\put(7.3,15){\makebox(0,0)[cc]{\tiny{$b$}}}
\put(5.3,17){\makebox(0,0)[cc]{\tiny{$c$}}}
\put(3.3,19){\makebox(0,0)[cc]{\tiny{$d$}}}
\end{picture}&
\begin{picture}(26,26)(0,0)
\put(12,0){\vector(0,3){25}} \put(0,12){\vector(3,0){25}}
\put(12.2,12){\circle*{.7}} \put(8.2,12){\circle*{.7}}
\put(12.2,8){\circle*{.7}} \put(8.2,8){\circle*{.7}}
\put(12.2,4){\circle*{.7}} \put(8.2,4){\circle*{.7}}
\put(8.2,12){\line(0,-1)4}\put(8.2,8){\line(0,-1)4}
\put(8.2,12){\line(1,-1)4}\put(8.2,8){\line(1,-1)4}
\put(8.2,8){\line(1,0)4}\put(8.2,4){\line(1,0)4}
\put(11.3,11){\makebox(0,0)[cc]{\tiny{$a$}}}
\put(9.3,9){\makebox(0,0)[cc]{\tiny{$b$}}}
\put(11.3,7){\makebox(0,0)[cc]{\tiny{$c$}}}
\put(9.3,5){\makebox(0,0)[cc]{\tiny{$d$}}}
\end{picture}\\  \hline
\end{tabular}$$
\caption{An orbit of blocks under the action of the symmetric group.}\end{table}
\end{example}

\begin{lemma}\label{ttt}
If $b$ is totally positive, for any $\sigma\in S_n$, we have $||[b]||\leq ||\sigma[b]||$, with equality if and only if
$\sigma[b]\in\{[b],[b]^*\}$, where 
 \begin{equation}\label{bst}
[b]^*=[-b_{n-1},\dots,-b_2,-b_1].\end{equation}
Furthermore,  the complex $\mathfrak{S}([b])=\mathfrak{S}_0\big(\sigma(\Da)\big)$, where
$\Da=\Da_0\cap\Da_{[b]}$ is a paralelotope whose edges are parallel to the  axes if and only if
$\sigma[b]\in\{[b],[b]^*\}$.
 \end{lemma}

 \begin{proof}
Note that if $[b]=\stackrel{\rightarrow}a+\langle\stackrel{\rightarrow}u\rangle$, the total
length of $[b]$ is the total variation of the sequence $\stackrel{\rightarrow}a=(a_1,\dots,a_n)$. The first inequality
 follows. Furthermore, the orbit of an element $[b]=[0,\dots,0,1,0,\dots,0]$ of length $1$ contains exactly two
vectors of that minimal length, namely the ones corresponding to an increasing and a decreasing sequence in the orbit
of $\stackrel{\rightarrow}a$. This proves that every complex $\mathfrak{S}(\sigma[b])$
in the orbit of $\mathfrak{S}([b])$, other than $\mathfrak{S}([b])$ or  $\mathfrak{S}([b]^*)$,
is a paralelotope whose edges are not parallel to the axes, so in particular $\sigma[b]$ is strictly larger than $[b]$.
 \end{proof}

Note that the correspondence $[b]\mapsto [b]^*$, as in (\ref{bst}), has the following properties:
\begin{itemize}
\item $\Da_{[b]^*}=\tau(\Da_{[b]})$, where $\tau=(1\ n)(2\ n-1)\cdots$
is the permutation reversing the $n$-tuple $(1,2,\dots,n)$,
\item the order $\Da_0\cap\Da_{[b]}$ is symmetric if and only if $[b]^*=-[b]$.
\end{itemize}
 The correspondence $\Ha\mapsto S_0(\Ha)$ reverse inclusions, so that for every pair
of elements $[c]$ and $[d]$ in $\Gamma$, with $\Da_{[c]},
\Da_{[d]}\in S_0(\Da)$, their intersection $\Da'=\Da_{[c]}\cap
\Da_{[d]}$ satisfies $S_0(\Da)\supseteq S_0(\Da')$. In fact, a stronger statement is true.

\begin{lemma}
Let $\Da=\Da_0\cap\Da_{[b]}$ be as above.
If $\Da_{[c]},
\Da_{[d]}\in S_0(\Da)$, satisfy $S_0(\Da)= S_0\big(\Da_{[c]}\cap
\Da_{[d]}\big)$, then $\left([c],[d]\right)=
(0,[b])$ or $\left([d],[c]\right)=
(0,[b])$.
\end{lemma}

 \begin{proof}
Observe that $\big|\big|[c]-[d]\big|\big|\leq \big|\big|[b]\big|\big|$, with equality if and only if
$\Da_{[c]}$ and $\Da_{[d]}$ are opposite vertices of $S_0(\Da)$. Conjugation by the diagonal matrix
$\mathrm{diag}(1,\pi^{c_1},\dots,\pi^{c_1+\dots+c_{n-1}})$ takes $\Da_{[t]}$ to $\Da_{[t]-[c]}$ for every $[t]\in\Gamma$. Note that there exists a permutation $\sigma\in S_n$ taking $[d]-[c]$ to a totally positive element $[r]$.
Since $S_n$ acts linearly on $\Gamma$, the cell-complex $\mathfrak{S}_0\left(\Da_{[c]}\cap\Da_{[d]}\right)$
is a parallelotope having $[c]$ and $[d]$ as opposite vertices. Furthermore, by hypotheses
$\mathfrak{S}_0\big(\Da_{[c]}\cap\Da_{[d]}\big)$ is a parallelotope whose edges are parallel to the coordinate axes. 
It follows from Lemma \ref{ttt} that $[d]-[c]\in\{[r],[r]^*\}$, since $[r]$ is totally positive. We conclude that either 
$[d]-[c]$ or $[c]-[d]$ is totally positive. The result follows.
\end{proof}

\begin{lemma}\label{pant}
Let $\Da$ be as above.
Let $\mu$ be an automorphism of $\alge$ satisfying $\mu\big(S_0(\Da)\big)= S_0\big(\Da)$. Then 
$$\left\{\mu(\Da_0),\mu\left(\Da_{[b]}\right)\right\}=
\left\{\Da_0,\Da_{[b]}\right\}.$$
\end{lemma}
 
 \begin{proof}
Note that $\mu(\Da)$ is contained in exactly the same maximal orders as $\Da$, 
and furthermore $\mu(\Da)=\mu(\Da_{0})\cap\mu(\Da_{[b]})$, whence the result
follows from the previous lemma.
\end{proof}

\begin{lemma}\label{ant}
There exists an automorphism of $\alge$, satisfying $\mu(\Da_{0})=\Da_{[b]}$ and 
$\mu(\Da_{[b]})=\Da_{0}$, if and only if $[b]^*=-[b]$.
\end{lemma}

 \begin{proof}
We denote by $\delta$ the canonical graph-distance on the 1-skeleton of the BT-building.
We say that a pair $(\Da,\Da')$ of maximal orders  is of line-type if there is exactly
$\delta(\Da,\Da')+1$ maximal orders containing $\Da\cap\Da'$. A pair $([c],[d])\in\Gamma^2$
is of line-type when $(\Da_{[c]},\Da_{[d]})$ is of line-type.
Automorphisms of $\alge$ necessarily take pairs of line-type to pairs of line-type.
Note that, if $[b]=[b_1,\dots,b_n]$ is totally positive, then 
$$1+\delta\left(\Da_0,\Da_{[b]}\right)\leq 1+\sum_{i=1}^{n-1}b_i\leq\prod_{i=1}^{n-1}(1+b_i),$$
an equality between the second and third expressions imply that at most one $b_i$ in non-zero.  
Note that the expression on the right of the preceding chain of inequalities is actually the number of vertices in $S_0(\Da)$.
Hence, if $([c],[d])$ is of line type and $[d]-[c]$ is totally positive, then
$S_0\left(\Da_{[c]}\cap_{[d]}\right)$ is a line parallel to one of the axes. We conclude that any any automorphism preserving  $\Da_{0}$, and mapping
$\Da_{[b]}$  to another totally positive order must preserve the set of axes of the polytope $S_0(\Da_0\cap\Da_{[b]})$.
No automorphism can take a line parallel to one axis to a line parallel to a different axis,
since the total distance  between consecutive elements in the line is diferent.
Necessity follows. 

If $[b]^*=-[b]$, the permutation $\rho\in S_n$ defined by $\rho(i)=n-i$ takes $[b]$ to
$-[b]$, so we can define $\mu=\tau\circ\rho$, where $\tau$ is conjugation by the diagonal matrix
$\mathrm{diag}(\pi^{a_1},\dots,\pi^{a_n})$ with 
$[b]=\stackrel{\rightarrow}a+\langle\stackrel{\rightarrow}u\rangle$. Sufficiency follows.
\end{proof}

Theorem \ref{t11} follows from the following lemma, which is an immediate consequence of
Lemma \ref{ant}.

\begin{lemma}
For any GEO $\Da=\Da_0\cap\Da_{[b]}$ in a $k$-CSA $\alge=\matrici_f(B)$, where $B$ is a division algebra, we have
$H_k(\Da)=H_k(\Da_0)=\oink_k^*k^{*f}$ unless the following conditios hold:
\begin{enumerate}
\item $[b]=-[b]^*$.
\item $\sum_{i=1}^na_i\equiv\frac{n}2\ (\mathrm{mod}\ n)$ for any (and therefore every) 
$\stackrel{\rightarrow}a\in\enteri^n$ satisfying
$\stackrel{\rightarrow}a+\langle\stackrel{\rightarrow}u\rangle=[b]$.
\end{enumerate}
In the latter case $H_k(\Da)=\oink_k^*k^{*(f/2)}$
\end{lemma}

\begin{proof}
It follows from Lemma \ref{pant} and Lemma \ref{ant} that condition 1 is necessary. When this is the case, we conclude, again from Lemma \ref{pant}, that $H(\Da)$ is generated by $\oink_k^*k^{*f}$ and the norm of an arbitrary
element $u$ satisfying $u\Da_0 u^{-1}=\Da_{[b]}$ and $u\Da_{[b]} u^{-1}=\Da_0$. The norm of such an element must
satisfy $N(u)\oink_k^*k^{*f}=\pi^d\oink_k^*k^{*f}$, where $d$ is the total distance between $\Da_0$ and $\Da_{[b]}$.
As noted in \S1, either $d\equiv0\ (\mathrm{mod}\ f)$ or $d\equiv\frac f2\ (\mathrm{mod}\ f)$. The result follows.
\end{proof}

\section{Representations}

Let $\Ha$ be a suborder of an order $\Da$ on $\alge$, let $\Sigma$ be the spinor class field for the 
genus of $\Da$, and consider the set $$\Phi=\{\rho(\Da,\Da')|\Da'\in\mathrm{gen}(\Da),\,\Ha\subseteq\Da'\}\subseteq\mathrm{Gal}(\Sigma/K).$$
When $\Phi$ is a group, the fixed subfield $F(\Da|\Ha)=\Sigma^\Phi$ is called the representation field. More generally,
the field  $F_-(\Da|\Ha)=\Sigma^{\langle\Phi\rangle}$, which is usually easy to compute, is called the 
lower representation field,
while the fixed field $F^-(\Da|\Ha)=\Sigma^{\Gamma}$, where $\Gamma=\{\gamma\in \mathrm{Gal}(\Sigma/K)|
\gamma \Phi=\Phi\}$, the upper representation field, has the important bound $F^-(\Da|\Ha)\subseteq L$ 
when $\Ha$ is an order contained in the maximal subfield $L$. Note that the representation field is defined if
and only if $\Gamma=\langle\Phi\rangle$, or equivalently, if $F_-(\Da|\Ha)=F^-(\Da|\Ha)$.

Let $$I(\Da|\Ha)=\{N(a)|a\in\alge_\ad,\ \Ha\subseteq a\Da a^{-1}\},$$ be the relative spinor image, then
the lower representation field is the class field corresponding to the set
$K^*\big\langle I(\Da|\Ha)\big\rangle$, while the upper representation field is the class 
field corresponding to the set $K^*H_-(\Da|\Ha)$, where 
$$H_-(\Da|\Ha)=\big\{a\in J_K|a I(\Da|\Ha)=I(\Da|\Ha)\big\}.$$ We conclude that
the functions $\Da\mapsto F_-(\Da|\Ha)$ and $\Da\mapsto F^-(\Da|\Ha)$ reverse inclusions.
In particular, if $\Ha$ is an order in a maximal subfield $L$, and  if we have $F_-(\Da|\Ha)=L$ for some order $\Da$ containing $\Ha$, the same holds
for every $\Da'$ with $\Ha\subseteq\Da'\subseteq\Da$. For the ring of integers $\oink_L$
of the maximal subfield $L$, we can give a more precise result. We say that an order $\Da$ of maximal
rank is strongly un-ramified if $\Sigma(\Da)\subseteq\Sigma(\Da_0)$ for some (or equivalently, any) maximal
order $\Da_0$. Recall that a $K$-CSA has no partial ramification if it is locally a matrix or a division algebra at all finite places.

\begin{proposition}\label{su}
Assume $\Da$ is a strongly unramified order, and $\alge$ has no partial ramification. Then for every maximal subfield
$L\subseteq\alge$, such that $\oink_L\subseteq \Da$, the representation field $F(\Da|\oink_L)$ is defined and in fact
$F(\Da|\oink_L)=\Sigma(\Da)\cap L$.
\end{proposition}

\begin{proof}
It follows from \cite[Prop. 4.3.4]{spinor} than $F(\Da_0|\oink_L)=\Sigma(\Da_0)\cap L$, 
if $\oink_L\subseteq\Da_0$ and $\Da_0$ is maximal.
The preceding paragraph implies  that $F_-(\Da|\oink_L)\supseteq\Sigma(\Da_0)\cap L$ for any order $\Da$ of maximal rank containing $\oink_L$. On the other hand, we always have $F^-(\Da|\oink_L)\subseteq  L$, since the group $H_-(\Da|\Ha)$
contains the group of norms $N_{L/K}(J_L)$, of $J_L$ identified as a subgroup of $\alge_ \ad^*$.
Since $\Sigma(\Da)\subseteq\Sigma(\Da_0)$, we have
$$\Sigma(\Da)\cap L\subseteq\Sigma(\Da_0)\cap L\subseteq F_-(\Da|\oink_L)\subseteq F^-(\Da|\oink_L)\subseteq \Sigma(\Da)\cap L,$$
whence equality follows.
\end{proof}

The preceding proposition applies in particular to GEOs. The hypothesis on $\alge$ is necessary, as shown by
the counter-example in \cite[\S4.3]{spinor}, as maximal orders are GEOs. This result does not helps us to know
whether $\oink_L$ embeds in some order of the genus of $\Da$ or not. This is a local problem and can be answered in some cases by Proposition \ref{p42} bellow. 

\begin{proposition}\label{p42}
Let $\Ha$ be a local order such that $\mathfrak{S}_0(\Ha)$ is contained in the standard apartment.
Assume $\Da$ is a local GEO of type $[b]=[b_1,\dots,b_{n-1}]$. Then a local order $\Ha$ is contained in a conjugate of
$\Da$ if and only if any of the following equivalent relations holds:
\begin{enumerate}
\item $\mathfrak{S}_0(\Ha)$ has two vertices whose type difference is $b$.
 \item There exist two vertices $\Da_{[c]}$ and $\Da_{[d]}$ in $S_0(\Ha)$, such that
$[d]-[c]$ is in the $S_n$-orbit of $[b]$.
\end{enumerate}
\end{proposition}

\begin{proof}
It suffices to prove the equivalence between both statements. It is immediate that $(2)$ implies $(1)$, so
we prove the converse. Assume that $S_0(\Ha)$ has two vertices whose type difference is $b$. By applying the action of $S_n$ on the BT-tree leaving the standard apartment invariant (\S3), we can assume that $[d]-[c]$ is in the first 
quadrant. Since $[b]$ is the type 
distance between $\Da_{[c]}$ and $\Da_{[d]}$, and therefore also the type distance between their conjugates
 $\Da_{0}$ and $\Da_{[d]-[c]}$, we conclude that $[d]-[c]=[b]$. The result follows.
\end{proof}

\begin{example}
Let $\Da$ and $\Ha$ be local GEOs of type $[3,1]$ and $[4,1]$ respectively. Then $\mathfrak{S}_0(\Da)$ can be embedded into
 $\mathfrak{S}_0(\Ha)$ in three different ways, corresponding to three embeddings of $\Ha$ into $\Da$, as shown in Figure 1A. Note that $A$ denote the image of one fixed maximal order in $\mathfrak{S}_0(\Da)$. In this case the relative spinor image is $I_k(\Da|\Ha)=k^*$, since the images of the vertex $A$, under each of the three embeddings, are at different total distances from the origin.
\begin{figure}
\[ (A)
\unitlength .6mm 
\linethickness{0.4pt}
\ifx\plotpoint\undefined\newsavebox{\plotpoint}\fi 
\begin{picture}(102.75,9.5)(0,0)
\put(14,1){\line(0,1){8}}
\put(14,9){\line(1,0){6}}
\put(14,1){\line(1,0){6}}
\put(20,1){\line(0,1){8}}
\put(20,9){\line(1,0){6}}
\put(20,1){\line(1,0){6}}
\put(26,1){\line(0,1){8}}
\put(26,9){\line(1,0){6}}
\put(26,1){\line(1,0){6}}
\put(32,1){\line(0,1){8}}
\multiput(14,9)(.0375,-.05){156}{\line(0,-1){.05}}
\multiput(20,9)(.0375,-.05){156}{\line(0,-1){.05}}
\multiput(26,9)(.0375,-.05){156}{\line(0,-1){.05}}
\multiput(32,9)(1,0){6}{{\rule{.4pt}{.4pt}}}
\multiput(38,9)(0,-1){8}{{\rule{.4pt}{.4pt}}}
\multiput(32,1)(1,0){6}{{\rule{.4pt}{.4pt}}}
\multiput(32,9)(.5,-.66){12}{{\rule{.4pt}{.4pt}}}
\put(14,1){\makebox(0,0)[cc]{$\bullet$}}
\put(11,0){\makebox(0,0)[cc]{${}_A$}}
\put(54,1){\line(0,1){8}}
\put(54,9){\line(1,0){6}}
\put(54,1){\line(1,0){6}}
\put(60,1){\line(0,1){8}}
\put(60,9){\line(1,0){6}}
\put(60,1){\line(1,0){6}}
\put(66,1){\line(0,1){8}}
\put(66,9){\line(1,0){6}}
\put(66,1){\line(1,0){6}}
\put(72,1){\line(0,1){8}}
\multiput(54,9)(.0375,-.05){156}{\line(0,-1){.05}}
\multiput(60,9)(.0375,-.05){156}{\line(0,-1){.05}}
\multiput(66,9)(.0375,-.05){156}{\line(0,-1){.05}}
\multiput(48,9)(1,0){6}{{\rule{.4pt}{.4pt}}}
\multiput(48,9)(0,-1){8}{{\rule{.4pt}{.4pt}}}
\multiput(48,1)(1,0){6}{{\rule{.4pt}{.4pt}}}
\multiput(48,9)(.5,-.66){12}{{\rule{.4pt}{.4pt}}}
\put(54,1){\makebox(0,0)[cc]{$\bullet$}}
\put(51,0){\makebox(0,0)[cc]{${}_A$}}
\put(84,9){\line(1,0){6}}
\put(90,1){\line(0,1){8}}
\put(90,9){\line(1,0){6}}
\put(90,1){\line(1,0){6}}
\put(96,1){\line(0,1){8}}
\put(96,9){\line(1,0){6}}
\put(96,1){\line(1,0){6}}
\put(102,1){\line(0,1){8}}
\put(102,1){\line(1,0){6}}
\multiput(84,9)(.0375,-.05){156}{\line(0,-1){.05}}
\multiput(90,9)(.0375,-.05){156}{\line(0,-1){.05}}
\multiput(96,9)(.0375,-.05){156}{\line(0,-1){.05}}
\multiput(102,9)(.0375,-.05){156}{\line(0,-1){.05}}
\multiput(84,1)(1,0){6}{{\rule{.4pt}{.4pt}}}
\multiput(102,9)(1,0){6}{{\rule{.4pt}{.4pt}}}
\multiput(84,9)(0,-1){8}{{\rule{.4pt}{.4pt}}}
\multiput(108,9)(0,-1){8}{{\rule{.4pt}{.4pt}}}
\put(84,9){\makebox(0,0)[cc]{$\bullet$}}
\put(81,8){\makebox(0,0)[cc]{${}_A$}}
\end{picture}
\]
\[(B)
\unitlength .6mm 
\linethickness{0.4pt}
\ifx\plotpoint\undefined\newsavebox{\plotpoint}\fi 
\begin{picture}(80.75,17.5)(0,0)
\put(14,1){\line(0,1){8}}
\put(14,9){\line(1,0){6}}
\put(14,1){\line(1,0){6}}
\put(20,1){\line(0,1){8}}
\put(20,9){\line(1,0){6}}
\put(20,1){\line(1,0){6}}
\put(26,1){\line(0,1){8}}
\put(26,9){\line(1,0){6}}
\put(26,1){\line(1,0){6}}
\put(32,1){\line(0,1){8}}
\put(14,9){\line(0,1){8}}
\put(14,17){\line(1,0){6}}
\put(20,9){\line(0,1){8}}
\put(20,17){\line(1,0){6}}
\put(26,9){\line(0,1){8}}
\put(26,17){\line(1,0){6}}
\put(26,17){\line(1,0){6}}
\put(32,9){\line(0,1){8}}
\multiput(14,17)(.0375,-.05){156}{\line(0,-1){.05}}
\multiput(20,17)(.0375,-.05){156}{\line(0,-1){.05}}
\multiput(26,17)(.0375,-.05){156}{\line(0,-1){.05}}
\multiput(14,9)(.0375,-.05){156}{\line(0,-1){.05}}
\multiput(20,9)(.0375,-.05){156}{\line(0,-1){.05}}
\multiput(26,9)(.0375,-.05){156}{\line(0,-1){.05}}
\multiput(32,17)(1,0){6}{{\rule{.4pt}{.4pt}}}
\multiput(38,17)(0,-1){8}{{\rule{.4pt}{.4pt}}}
\multiput(32,9)(1,0){6}{{\rule{.4pt}{.4pt}}}
\multiput(32,17)(.5,-.66){12}{{\rule{.4pt}{.4pt}}}
\multiput(32,1)(1,0){6}{{\rule{.4pt}{.4pt}}}
\multiput(38,9)(0,-1){8}{{\rule{.4pt}{.4pt}}}
\multiput(32,9)(.5,-.66){12}{{\rule{.4pt}{.4pt}}}
\put(14,1){\makebox(0,0)[cc]{$\bullet$}}
\put(11,0){\makebox(0,0)[cc]{${}_A$}}
\put(54,9){\line(0,1){8}}
\put(54,17){\line(1,0){6}}
\put(54,9){\line(1,0){6}}
\put(60,9){\line(0,1){8}}
\put(60,17){\line(1,0){6}}
\put(60,9){\line(1,0){6}}
\put(66,9){\line(0,1){8}}
\put(66,17){\line(1,0){6}}
\put(66,9){\line(1,0){6}}
\put(72,9){\line(0,1){8}}
\put(54,1){\line(0,1){8}}
\put(54,1){\line(1,0){6}}
\put(60,1){\line(0,1){8}}
\put(60,1){\line(1,0){6}}
\put(66,1){\line(0,1){8}}
\put(66,1){\line(1,0){6}}
\put(72,1){\line(0,1){8}}
\multiput(54,17)(.0375,-.05){156}{\line(0,-1){.05}}
\multiput(60,17)(.0375,-.05){156}{\line(0,-1){.05}}
\multiput(66,17)(.0375,-.05){156}{\line(0,-1){.05}}
\multiput(54,9)(.0375,-.05){156}{\line(0,-1){.05}}
\multiput(60,9)(.0375,-.05){156}{\line(0,-1){.05}}
\multiput(66,9)(.0375,-.05){156}{\line(0,-1){.05}}
\multiput(48,17)(1,0){6}{{\rule{.4pt}{.4pt}}}
\multiput(48,17)(0,-1){8}{{\rule{.4pt}{.4pt}}}
\multiput(48,9)(1,0){6}{{\rule{.4pt}{.4pt}}}
\multiput(48,17)(.5,-.66){12}{{\rule{.4pt}{.4pt}}}
\multiput(48,9)(0,-1){8}{{\rule{.4pt}{.4pt}}}
\multiput(48,1)(1,0){6}{{\rule{.4pt}{.4pt}}}
\multiput(48,9)(.5,-.66){12}{{\rule{.4pt}{.4pt}}}
\put(54,1){\makebox(0,0)[cc]{$\bullet$}}
\put(51,0){\makebox(0,0)[cc]{${}_A$}}
\end{picture}
\]
\caption{Two examples of embeddings of one GEO into another.}
\end{figure}
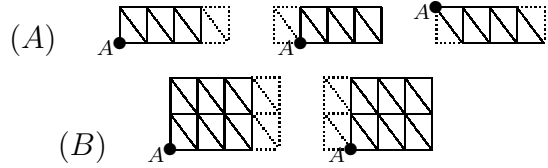
\end{example}

\begin{example}
Let $\Da$ and $\Ha$ be local GEOs of type $[3,2]$ and $[4,2]$ respectively. Then $\mathfrak{S}_0(\Da)$ can be embedded into $\mathfrak{S}_0(\Ha)$ in just two different ways (Figure 1B). In this case the relative spinor image $I_k(\Da|\Ha)=\oink_k^*k^{*3}\cup \pi^2\oink_k^*k^{*3}$ is not a group.
\end{example}

In order to apply this result to commutative orders, we need an explicit description of the cell complex $\mathfrak{S}_0(\Ha)$. We can do this for the order $\Ha=\oink_L$,
for a maximal separable subfield $L\subseteq\alge$, when $\alge$ has no partial ramification. Recall that, for a global separable field extension $L/K$, the local completion
$L_\wp=L\otimes_KK_\wp$ is a product of fields, and therefore we need a description of $\mathfrak{S}_0(\oink_L)$ for separable
commutative algebras $L$. This is provided by next result. Here we identify the set of vertices of an $n$--dimensional apartment with
$\enteri^n$, and all cartesian products must be understood in this context. 

\begin{proposition}
Assume that the $n$-dimensional local algebra $L=\prod_{i=1}^rL_i\subseteq\matrici_n(k)$ is a product of fields. Then 
$\mathfrak{S}_0(\oink_L)$ is contained in an apartment $A$ and its set of vertices has a decomposition of the form
$$S_0(\oink_L)=S_1\times\enteri\times S_2\times\cdots\times\enteri\times S_r,$$
where every $S_i\subseteq\enteri^{\dim_kL_i-1}$ is the set of vertices in a simplex of dimension $e(L_i/k)-1$.
\end{proposition}

\begin{proof}
Note that the regular representation $\phi:L\rightarrow\matrici_n(k)$ is, up to conjugacy, the only 
faithfull $n$-dimensional representation of the $k$-algebra $L$. We conclude that the maximal orders containing
$\oink_L$ are in correspondence with the classes of fractional ideals in $\oink_L$ up to $k^*$-multiplication.
By choosing a suitable basis, we can assume that $L$ is the set of matrices of the form
$$\left(\begin{array}{cccc}\phi_1(\lambda_1)&0&\cdots&0\\0&\phi_2(\lambda_2)&\cdots&0\\
 \vdots&\vdots&\ddots&\vdots\\ 0&0&\cdots&\phi_r(\lambda_r)\end{array}\right),\qquad \lambda_i\in L_i,$$
where $\phi_i:L_i\rightarrow\matrici_{[L_i:k]}(k)$ is the regular representation with respect to a basis
of the form $S\cup\pi_iS\cup\cdots\cup\pi^{e(L_i/K)-1}S$ of $L_i$, for an arbitrary $k$-basis $S$ 
of the largest unramified subfield of $L_i$. Then all fractional ideals in $L$ have the form
$I_1\times\cdots\times I_r$, where $I_i$ is a fractional ideal in $L_i$. In particular, $I_i$ is is homothetic to one of the ideals
$(1),\left(\pi_i^1\right),\dots,\left(\pi_i^{e(L_i/K)-1}\right)$. The fact that the corresponding vertices form a simplex is immediate from the definition of the BT-building.
\end{proof}
\begin{figure}
\unitlength 1mm 
\linethickness{0.4pt}
\ifx\plotpoint\undefined\newsavebox{\plotpoint}\fi 
\[(A)
\begin{picture}(26,26)(0,0)
\put(12,0){\vector(0,3){25}} \put(0,12){\vector(3,0){25}}
\put(16.2,8){\line(0,-1)4}\put(16.2,12){\line(0,-1)4}
\put(16.2,16){\line(0,-1)4}\put(16.2,20){\line(0,-1)4}
\put(12.2,8){\line(1,0)4}\put(12.2,4){\line(1,0)4}
\put(12.2,16){\line(1,0)4}\put(12.2,20){\line(1,0)4}
\put(12.2,12){\line(1,-1)4}\put(12.2,8){\line(1,-1)4}
\put(12.2,16){\line(1,-1)4}\put(12.2,20){\line(1,-1)4}
\put(14.2,24){\makebox(0,0)[cc]{\vdots}}\put(14.2,2){\makebox(0,0)[cc]{\vdots}}
\end{picture}\qquad(B)
\begin{picture}(86,26)(0,0)
\put(12,0){\vector(0,3){25}} \put(0,12){\vector(3,0){25}}
\put(16.2,16){\line(0,-1)4}\put(16.2,20){\line(0,-1)4}
\put(12.2,16){\line(1,0)4}\put(12.2,20){\line(1,0)4}
\put(12.2,16){\line(1,-1)4}\put(12.2,20){\line(1,-1)4}
\put(42,0){\vector(0,3){25}} \put(30,12){\vector(3,0){25}}
\put(46.2,12){\line(0,-1)4}
\put(46.2,16){\line(0,-1)4}
\put(42.2,16){\line(1,0)4}
\put(42.2,12){\line(1,-1)4}
\put(42.2,16){\line(1,-1)4}\put(42.2,20){\line(1,-1)4}
\put(72,0){\vector(0,3){25}} \put(60,12){\vector(3,0){25}}
\put(76.2,12){\line(0,-1)4}\put(76.2,16){\line(0,-1)4}
\put(72.2,8){\line(1,0)4}\put(72.2,16){\line(1,0)4}
\put(72.2,12){\line(1,-1)4}\put(72.2,16){\line(1,-1)4}
\put(12,12){\makebox(0,0)[cc]{$\bullet$}}
\put(46,8){\makebox(0,0)[cc]{$\bullet$}}
\put(72,8){\makebox(0,0)[cc]{$\bullet$}}
\end{picture}
\] \caption{Embedding a maximal conmutative order in a GEO of type $[1,2]$.}
\end{figure}
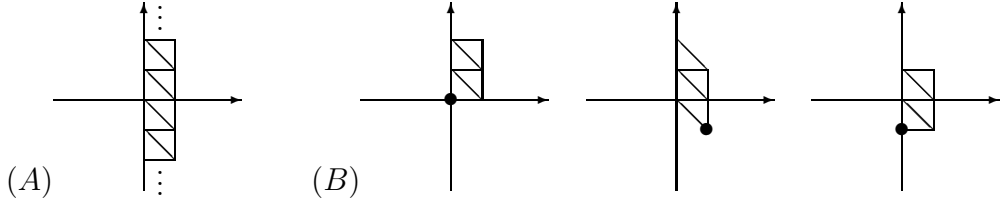

\begin{example}
Consider the algebra $L=F\times k$, where $F$ is a ramified quadratic extension of
$k$, identified with the set of matrices of the form $\bbmatrix {\phi(f)}00a$, for $f\in F$, $a\in k$, where $\phi:F\rightarrow \matrici_2(k)$ is the regular representation. Then $\mathfrak{S}_0(\oink_L)$ is as shown in Figure 2A.

The picture already tell us that $\oink_L$ embeds into a local  GEO $\Da$ of type $[1,2]$, namely the one corresponding to the block in the left of Figure 2B. The blocks of the orders $E\Da E^{-1}$ and $E^2\Da E^{-2}$, where
$E=\bbmatrix{\phi(\pi_F)}001$ and $\pi_F$ is a uniformizing parameter of $F$, are shown in the other images of Figure 2B.
Note that the multiplicative group of $L$ acts transitively on the set of fractional ideals, whence conjugating by such
elements, the block of any GEO representing $\oink_L$ can be moved inside $\mathfrak{S}_0(\oink_L)$, taking a given vertex to any prescribed possition in this block. This can be used to give a second proof of Proposition \ref{su} for GEOs.
\end{example}

\section{global description of GEOs}

In the quaternionic case,there is a simple way to describe the spinor genera in a genus of Eichler orders in terms of the set of spinor genera of maximal orders. Here we describe it for an Eichler order $\Da$ whose level has only two prime divisors, say $\mathcal{L}(\Da)=
\wp_1^{\alpha_1}\wp_2^{\alpha_2}$ and leave the obvious generalization to the reader.
 Note that there exists $\alpha_i+1$ local maximal orders containing
 $\Da_{\wp_i}$ and they lie on a path of the BT-tree at $\wp_i$ for $i\in\{1,2\}$. It follows that the global maximal orders containing $\Da$ correspond to the vertices of a rectangular grid
with $\alpha_1+1$ columns and $\alpha_2+1$ rows. If we label these vertices alternating labels on each row and column as shown in Figure 3A, each label correspond to a spinor genus. If condition \textbf{GEC} holds, equally labeled vertices correspond to isomorphic maximal orders. 

Let $\mathbb{O}_0$ be the genus of maximal orders in $\alge$ and $\rho_0:\mathbb{O}_0\times\mathbb{O}_0
\rightarrow\mathrm{Gal}(\Sigma_0/K)$ the distance map on maximal orders.
The maximal orders $\Da_1$ and $\Da_2$, corresponding to any pair of
horizontally (resp. vertically) adjacent vertices, satisfy $\rho_0(\Da_1,\Da_2)=|[\wp_1,\Sigma_0/K]|$ (resp. $|[\wp_2,\Sigma/K]|$), where $I\mapsto|[I,\Sigma_0/K]|$ is the Artin map on ideals. It follows that the isomorphism class of every maximal order containing $\Da$ depends only
on the isomorphism class of the order in the lower-left corner. 
This gives a simple way to describe which eichler orders embed into which others whose precise formulation is left to the reader. 
From the description of the spinor class field for Eichler orders of a given level, quoted in the introduction, next result follows:
\begin{proposition}\label{eo}
Let $\Da=\Da_1\cap\Da_2$ and $\Da'=\Da_3\cap\Da_4$ be Eichler orders of the same level, and 
let $\Sigma\subseteq\Sigma_0$ be the corresponding spinor class field. Then $\Da$ and $\Da'$ are in the
same spinor genus if and only if $\rho_0(\Da_1,\Da_3)$ is trivial on $\Sigma$. 
\end{proposition}

\begin{example}
Assume $|[\wp_1,\Sigma_0/K]|$ and $|[\wp_2,\Sigma_0/K]|$ are non-trivial and different. If $\alpha_1=3$ and
$\alpha_2=5$ as in Figure 3A, we have $[\Sigma_0:\Sigma]=4$, and $\Da_1\cap\Da_2$ is in the same spinor genus as
 $\Da_3\cap\Da_4$, as soon as $\Da_1$ is in the spinor genus $A$, while $\Da_3\in A\cup B\cup C\cup D$. 
In Figure 3C, the level $\wp_1^2\wp_2^4$ of the Eichler order is a square, so there is a unique spinor genus of
maximal orders, namely $A$, corresponding to this Eichler order. 
Figure 3B shows an example where $|[\wp_1,\Sigma_0/K]|=|[\wp_2,\Sigma_0/K]|$ is non-trivial, whence 
$[\Sigma_0:\Sigma]=2$. Here the same Eichler order can be defined as the intersection of two orders in $A$ or two orders
in $C$.
\begin{figure}
\unitlength 1mm 
\linethickness{0.4pt}
\ifx\plotpoint\undefined\newsavebox{\plotpoint}\fi 
\[\textnormal{(A)}\ 
\begin{picture}(28,28)(0,0)
\put(0,0){\line(0,1)5}\put(0,5){\line(0,1)5}
\put(0,10){\line(0,1)5}\put(0,15){\line(0,1)5}
\put(0,20){\line(0,1)5}
\put(0,0){\line(1,0)5}\put(0,5){\line(1,0)5}
\put(0,10){\line(1,0)5}\put(0,15){\line(1,0)5}
\put(0,20){\line(1,0)5}\put(0,25){\line(1,0)5}
\put(5,0){\line(0,1)5}\put(5,5){\line(0,1)5}
\put(5,10){\line(0,1)5}\put(5,15){\line(0,1)5}
\put(5,20){\line(0,1)5}
\put(5,0){\line(1,0)5}\put(5,5){\line(1,0)5}
\put(5,10){\line(1,0)5}\put(5,15){\line(1,0)5}
\put(5,20){\line(1,0)5}\put(5,25){\line(1,0)5}
\put(10,0){\line(0,1)5}\put(10,5){\line(0,1)5}
\put(10,10){\line(0,1)5}\put(10,15){\line(0,1)5}
\put(10,20){\line(0,1)5}
\put(10,0){\line(1,0)5}\put(10,5){\line(1,0)5}
\put(10,10){\line(1,0)5}\put(10,15){\line(1,0)5}
\put(10,20){\line(1,0)5}\put(10,25){\line(1,0)5}
\put(15,0){\line(0,1)5}\put(15,5){\line(0,1)5}
\put(15,10){\line(0,1)5}\put(15,15){\line(0,1)5}
\put(15,20){\line(0,1)5}
\put(0,0){\makebox(0,0)[cc]{$\bullet$}}\put(0,5){\makebox(0,0)[cc]{$\bullet$}}
\put(0,10){\makebox(0,0)[cc]{$\bullet$}}\put(0,15){\makebox(0,0)[cc]{$\bullet$}}
\put(0,20){\makebox(0,0)[cc]{$\bullet$}}\put(0,25){\makebox(0,0)[cc]{$\bullet$}}
\put(5,0){\makebox(0,0)[cc]{$\bullet$}}\put(5,5){\makebox(0,0)[cc]{$\bullet$}}
\put(5,10){\makebox(0,0)[cc]{$\bullet$}}\put(5,15){\makebox(0,0)[cc]{$\bullet$}}
\put(5,20){\makebox(0,0)[cc]{$\bullet$}}\put(5,25){\makebox(0,0)[cc]{$\bullet$}}
\put(10,0){\makebox(0,0)[cc]{$\bullet$}}\put(10,5){\makebox(0,0)[cc]{$\bullet$}}
\put(10,10){\makebox(0,0)[cc]{$\bullet$}}\put(10,15){\makebox(0,0)[cc]{$\bullet$}}
\put(10,20){\makebox(0,0)[cc]{$\bullet$}}\put(10,25){\makebox(0,0)[cc]{$\bullet$}}
\put(15,0){\makebox(0,0)[cc]{$\bullet$}}\put(15,5){\makebox(0,0)[cc]{$\bullet$}}
\put(15,10){\makebox(0,0)[cc]{$\bullet$}}\put(15,15){\makebox(0,0)[cc]{$\bullet$}}
\put(15,20){\makebox(0,0)[cc]{$\bullet$}}\put(15,25){\makebox(0,0)[cc]{$\bullet$}}
\put(1,2){\makebox(0,0)[cc]{${}_A$}}\put(1,7){\makebox(0,0)[cc]{${}_B$}}
\put(1,12){\makebox(0,0)[cc]{${}_A$}}\put(1,17){\makebox(0,0)[cc]{${}_B$}}
\put(1,22){\makebox(0,0)[cc]{${}_A$}}\put(1,27){\makebox(0,0)[cc]{${}_B$}}
\put(6,2){\makebox(0,0)[cc]{${}_C$}}\put(6,7){\makebox(0,0)[cc]{${}_D$}}
\put(6,12){\makebox(0,0)[cc]{${}_C$}}\put(6,17){\makebox(0,0)[cc]{${}_D$}}
\put(6,22){\makebox(0,0)[cc]{${}_C$}}\put(6,27){\makebox(0,0)[cc]{${}_D$}}
\put(11,2){\makebox(0,0)[cc]{${}_A$}}\put(11,7){\makebox(0,0)[cc]{${}_B$}}
\put(11,12){\makebox(0,0)[cc]{${}_A$}}\put(11,17){\makebox(0,0)[cc]{${}_B$}}
\put(11,22){\makebox(0,0)[cc]{${}_A$}}\put(11,27){\makebox(0,0)[cc]{${}_B$}}
\put(16,2){\makebox(0,0)[cc]{${}_C$}}\put(16,7){\makebox(0,0)[cc]{${}_D$}}
\put(16,12){\makebox(0,0)[cc]{${}_C$}}\put(16,17){\makebox(0,0)[cc]{${}_D$}}
\put(16,22){\makebox(0,0)[cc]{${}_C$}}\put(16,27){\makebox(0,0)[cc]{${}_D$}}
\end{picture}\qquad\textnormal{(B)}\ 
\begin{picture}(28,28)(0,0)
\put(0,0){\line(0,1)5}\put(0,5){\line(0,1)5}
\put(0,10){\line(0,1)5}\put(0,15){\line(0,1)5}
\put(0,20){\line(0,1)5}
\put(0,0){\line(1,0)5}\put(0,5){\line(1,0)5}
\put(0,10){\line(1,0)5}\put(0,15){\line(1,0)5}
\put(0,20){\line(1,0)5}\put(0,25){\line(1,0)5}
\put(5,0){\line(0,1)5}\put(5,5){\line(0,1)5}
\put(5,10){\line(0,1)5}\put(5,15){\line(0,1)5}
\put(5,20){\line(0,1)5}
\put(5,0){\line(1,0)5}\put(5,5){\line(1,0)5}
\put(5,10){\line(1,0)5}\put(5,15){\line(1,0)5}
\put(5,20){\line(1,0)5}\put(5,25){\line(1,0)5}
\put(10,0){\line(0,1)5}\put(10,5){\line(0,1)5}
\put(10,10){\line(0,1)5}\put(10,15){\line(0,1)5}
\put(10,20){\line(0,1)5}
\put(10,0){\line(1,0)5}\put(10,5){\line(1,0)5}
\put(10,10){\line(1,0)5}\put(10,15){\line(1,0)5}
\put(10,20){\line(1,0)5}\put(10,25){\line(1,0)5}
\put(15,0){\line(0,1)5}\put(15,5){\line(0,1)5}
\put(15,10){\line(0,1)5}\put(15,15){\line(0,1)5}
\put(15,20){\line(0,1)5}
\put(0,0){\makebox(0,0)[cc]{$\bullet$}}\put(0,5){\makebox(0,0)[cc]{$\bullet$}}
\put(0,10){\makebox(0,0)[cc]{$\bullet$}}\put(0,15){\makebox(0,0)[cc]{$\bullet$}}
\put(0,20){\makebox(0,0)[cc]{$\bullet$}}\put(0,25){\makebox(0,0)[cc]{$\bullet$}}
\put(5,0){\makebox(0,0)[cc]{$\bullet$}}\put(5,5){\makebox(0,0)[cc]{$\bullet$}}
\put(5,10){\makebox(0,0)[cc]{$\bullet$}}\put(5,15){\makebox(0,0)[cc]{$\bullet$}}
\put(5,20){\makebox(0,0)[cc]{$\bullet$}}\put(5,25){\makebox(0,0)[cc]{$\bullet$}}
\put(10,0){\makebox(0,0)[cc]{$\bullet$}}\put(10,5){\makebox(0,0)[cc]{$\bullet$}}
\put(10,10){\makebox(0,0)[cc]{$\bullet$}}\put(10,15){\makebox(0,0)[cc]{$\bullet$}}
\put(10,20){\makebox(0,0)[cc]{$\bullet$}}\put(10,25){\makebox(0,0)[cc]{$\bullet$}}
\put(15,0){\makebox(0,0)[cc]{$\bullet$}}\put(15,5){\makebox(0,0)[cc]{$\bullet$}}
\put(15,10){\makebox(0,0)[cc]{$\bullet$}}\put(15,15){\makebox(0,0)[cc]{$\bullet$}}
\put(15,20){\makebox(0,0)[cc]{$\bullet$}}\put(15,25){\makebox(0,0)[cc]{$\bullet$}}
\put(1,2){\makebox(0,0)[cc]{${}_A$}}\put(1,7){\makebox(0,0)[cc]{${}_C$}}
\put(1,12){\makebox(0,0)[cc]{${}_A$}}\put(1,17){\makebox(0,0)[cc]{${}_C$}}
\put(1,22){\makebox(0,0)[cc]{${}_A$}}\put(1,27){\makebox(0,0)[cc]{${}_C$}}
\put(6,2){\makebox(0,0)[cc]{${}_C$}}\put(6,7){\makebox(0,0)[cc]{${}_A$}}
\put(6,12){\makebox(0,0)[cc]{${}_C$}}\put(6,17){\makebox(0,0)[cc]{${}_A$}}
\put(6,22){\makebox(0,0)[cc]{${}_C$}}\put(6,27){\makebox(0,0)[cc]{${}_A$}}
\put(11,2){\makebox(0,0)[cc]{${}_A$}}\put(11,7){\makebox(0,0)[cc]{${}_C$}}
\put(11,12){\makebox(0,0)[cc]{${}_A$}}\put(11,17){\makebox(0,0)[cc]{${}_C$}}
\put(11,22){\makebox(0,0)[cc]{${}_A$}}\put(11,27){\makebox(0,0)[cc]{${}_C$}}
\put(16,2){\makebox(0,0)[cc]{${}_C$}}\put(16,7){\makebox(0,0)[cc]{${}_A$}}
\put(16,12){\makebox(0,0)[cc]{${}_C$}}\put(16,17){\makebox(0,0)[cc]{${}_A$}}
\put(16,22){\makebox(0,0)[cc]{${}_C$}}\put(16,27){\makebox(0,0)[cc]{${}_A$}}
\end{picture}\qquad\textnormal{(C)}\ 
\begin{picture}(28,28)(0,0)
\put(0,0){\line(0,1)5}\put(0,5){\line(0,1)5}
\put(0,10){\line(0,1)5}\put(0,15){\line(0,1)5}
\put(0,0){\line(1,0)5}\put(0,5){\line(1,0)5}
\put(0,10){\line(1,0)5}\put(0,15){\line(1,0)5}
\put(0,20){\line(1,0)5}
\put(5,0){\line(0,1)5}\put(5,5){\line(0,1)5}
\put(5,10){\line(0,1)5}\put(5,15){\line(0,1)5}
\put(5,0){\line(1,0)5}\put(5,5){\line(1,0)5}
\put(5,10){\line(1,0)5}\put(5,15){\line(1,0)5}
\put(5,20){\line(1,0)5}
\put(10,0){\line(0,1)5}\put(10,5){\line(0,1)5}
\put(10,10){\line(0,1)5}\put(10,15){\line(0,1)5}
\put(0,0){\makebox(0,0)[cc]{$\bullet$}}\put(0,5){\makebox(0,0)[cc]{$\bullet$}}
\put(0,10){\makebox(0,0)[cc]{$\bullet$}}\put(0,15){\makebox(0,0)[cc]{$\bullet$}}
\put(0,20){\makebox(0,0)[cc]{$\bullet$}}
\put(5,0){\makebox(0,0)[cc]{$\bullet$}}\put(5,5){\makebox(0,0)[cc]{$\bullet$}}
\put(5,10){\makebox(0,0)[cc]{$\bullet$}}\put(5,15){\makebox(0,0)[cc]{$\bullet$}}
\put(5,20){\makebox(0,0)[cc]{$\bullet$}}
\put(10,0){\makebox(0,0)[cc]{$\bullet$}}\put(10,5){\makebox(0,0)[cc]{$\bullet$}}
\put(10,10){\makebox(0,0)[cc]{$\bullet$}}\put(10,15){\makebox(0,0)[cc]{$\bullet$}}
\put(10,20){\makebox(0,0)[cc]{$\bullet$}}
\put(1,2){\makebox(0,0)[cc]{${}_A$}}\put(1,7){\makebox(0,0)[cc]{${}_B$}}
\put(1,12){\makebox(0,0)[cc]{${}_A$}}\put(1,17){\makebox(0,0)[cc]{${}_B$}}
\put(1,22){\makebox(0,0)[cc]{${}_A$}}
\put(6,2){\makebox(0,0)[cc]{${}_C$}}\put(6,7){\makebox(0,0)[cc]{${}_D$}}
\put(6,12){\makebox(0,0)[cc]{${}_C$}}\put(6,17){\makebox(0,0)[cc]{${}_D$}}
\put(6,22){\makebox(0,0)[cc]{${}_C$}}
\put(11,2){\makebox(0,0)[cc]{${}_A$}}\put(11,7){\makebox(0,0)[cc]{${}_B$}}
\put(11,12){\makebox(0,0)[cc]{${}_A$}}\put(11,17){\makebox(0,0)[cc]{${}_B$}}
\put(11,22){\makebox(0,0)[cc]{${}_A$}}
\end{picture}
\] \caption{Maximal orders containing an Eichler order.}
\end{figure}
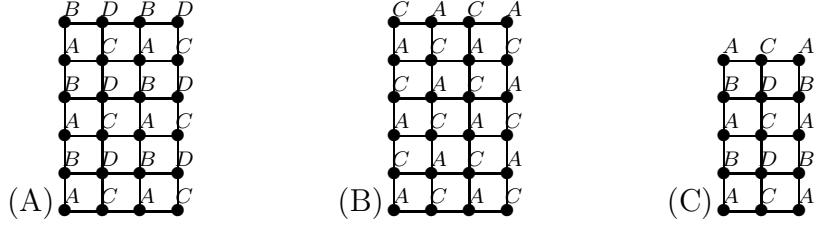
\end{example}

For a genus $\mathrm{gen}(\Da)$ of GEOs, we can give a similar labeling of vertices  of
the product cell complex $\mathbb{S}(\Da)=
\prod_{\wp}\mathfrak{S}_0(\Da,\wp)$ where the product is taken over all places at which
$\Da$ is not maximal, and each factor of this product is the block corresponding to a local GEO as described in \S3-\S4. When $\Da_\wp=\Da_0\cap\Da_{[r]}$ for a totally positive element $[r]\in\Gamma$,  the vertices corresponding to
$0$ and $[r]$ can be identified as the most distant pair (with respect to the usual distance in the $1$-skeleton of the complex) in that factor. We call them extreme vertices. A vertex in the product cell complex $\mathbb{S}(\Da)$ is called extreme if each of its coordinates is extreme.  In this case, Proposition \ref{eo} has a natural generalization whose proof is straightforward:
\begin{proposition}\label{geo}
Let $\Da=\Da_1\cap\Da_2$ and $\Da'=\Da_3\cap\Da_4$ be locally symmetric GEOs in the same genus, and 
let $\Sigma\subseteq\Sigma_0$ be the corresponding spinor class field. Then $\Da$ and $\Da'$ are in the
same spinor genus if and only if $\rho_0(\Da_1,\Da_3)$ is trivial on $\Sigma$. 
\end{proposition}
This result does not generalize to non-symmetric GEOs. In fact, when $\alge$ is odd-dimensional, we always have $\Sigma=\Sigma_0$, but we can, certainly, intersect maximal orders in different spinor genera. The spinor genus
of a GEO $\Da_1\cap\Da_2$, where $\Da_1$ and $\Da_2$ are maximal orders, depends not only on the spinor genus 
of $\Da_1$, but also on the relative positions of $\Da_1$ as a vertex in the complex.
This phenomenon can be illustrated by next example.

\begin{example}
Assume $\alge$ is a $6$-dimensional algebra having degree $2$ at $\wp_1$ and degree $3$ at $\wp_2$, and let $\Da$
be a GEO having type $[1,1]$ at $\wp_1$ and $[1]$ at $\wp_2$. Then the complex $\mathbb{S}(\Da)$ is a cube as shown
in Figure 4A. This is a locally symmetric GEO, whence the order $\Da'$ whose complex $\mathbb{S}(\Da')$ has an order in
the spinor genus $D$ in the lower left corner (Figure 4C) is isomorphic to $\Da$. On the other hand, if  $\Da''$
is a GEO in the same algebra having type $[1,2]$ at $\wp_1$ and $[1]$ at $\wp_2$, the complex $\mathbb{S}(\Da'')$ is as shown in Figure 4B. This is not locally symetric at $\wp_1$, and in fact the order $\Da'''$ whose complex $\mathbb{S}(\Da')$ has an order in the spinor genus $C$ in the lower left corner is not isomorphic to $\Da''$, as a quick glance to
Figure 4D shows.
\begin{figure}
\unitlength 1mm 
\linethickness{0.4pt}
\ifx\plotpoint\undefined\newsavebox{\plotpoint}\fi 
\[\textnormal{(A)}\ 
\begin{picture}(28,28)(0,0)
\put(0,0){\line(0,1){15}}\put(0,0){\line(1,0){15}}
\put(15,0){\line(0,1){15}}\put(0,15){\line(1,0){15}}
\put(0,15){\line(1,1)7}\put(15,15){\line(1,1)7}
\put(7,22){\line(1,0){15}}\put(22,7){\line(0,1){15}}
\put(0,15){\line(1,1)7}\put(15,15){\line(1,1)7}\put(15,0){\line(1,1)7}
\multiput(15,15)(-0.1,0.0875){80}{\line(0,-1){.05}}
\multiput(15,0)(-0.5,0.4375){16}{\line(0,-1){.05}}
\multiput(0,0)(0.4375,0.4375){16}{\line(0,-1){.05}}
\multiput(7,7)(0,0.5){30}{\line(0,1){.05}}
\multiput(7,7)(0.5,0){30}{\line(1,0){.05}}
\put(0,0){\makebox(0,0)[cc]{$\bullet$}}\put(7,7){\makebox(0,0)[cc]{$\bullet$}}
\put(22,7){\makebox(0,0)[cc]{$\bullet$}}\put(15,0){\makebox(0,0)[cc]{$\bullet$}}
\put(0,15){\makebox(0,0)[cc]{$\bullet$}}\put(7,22){\makebox(0,0)[cc]{$\bullet$}}
\put(22,22){\makebox(0,0)[cc]{$\bullet$}}\put(15,15){\makebox(0,0)[cc]{$\bullet$}}
\put(1,2){\makebox(0,0)[cc]{${}_A$}}\put(8,9){\makebox(0,0)[cc]{${}_B$}}
\put(23,9){\makebox(0,0)[cc]{${}_A$}}\put(16,2){\makebox(0,0)[cc]{${}_C$}}
\put(1,17){\makebox(0,0)[cc]{${}_D$}}\put(6,24){\makebox(0,0)[cc]{${}_E$}}
\put(23,24){\makebox(0,0)[cc]{${}_D$}}\put(16,17){\makebox(0,0)[cc]{${}_F$}}
\end{picture}
\qquad\textnormal{(B)}\ 
\begin{picture}(43,28)(0,0)
\put(0,0){\line(0,1){15}}\put(0,0){\line(1,0){15}}
\put(15,0){\line(0,1){15}}\put(0,15){\line(1,0){15}}
\put(0,15){\line(1,1)7}\put(15,15){\line(1,1)7}
\put(7,22){\line(1,0){15}}
\put(0,15){\line(1,1)7}\put(15,15){\line(1,1)7}\put(30,0){\line(1,1)7}
\put(15,0){\line(1,0){15}}
\put(30,0){\line(0,1){15}}\put(15,15){\line(1,0){15}}
\put(15,15){\line(1,1)7}\put(30,15){\line(1,1)7}
\put(22,22){\line(1,0){15}}\put(37,7){\line(0,1){15}}
\put(0,15){\line(1,1)7}\put(15,15){\line(1,1)7}
\multiput(15,15)(-0.1,0.0875){80}{\line(0,-1){.05}}
\multiput(30,15)(-0.1,0.0875){80}{\line(0,-1){.05}}
\multiput(15,0)(-0.5,0.4375){16}{\line(0,-1){.05}}
\multiput(0,0)(0.4375,0.4375){16}{\line(0,-1){.05}}
\multiput(7,7)(0,0.5){30}{\line(0,1){.05}}
\multiput(7,7)(0.5,0){30}{\line(1,0){.05}}
\multiput(30,0)(-0.5,0.4375){16}{\line(0,-1){.05}}
\multiput(15,0)(0.4375,0.4375){16}{\line(0,-1){.05}}
\multiput(22,7)(0,0.5){30}{\line(0,1){.05}}
\multiput(22,7)(0.5,0){30}{\line(1,0){.05}}
\put(0,0){\makebox(0,0)[cc]{$\bullet$}}\put(7,7){\makebox(0,0)[cc]{$\bullet$}}
\put(22,7){\makebox(0,0)[cc]{$\bullet$}}\put(15,0){\makebox(0,0)[cc]{$\bullet$}}
\put(0,15){\makebox(0,0)[cc]{$\bullet$}}\put(7,22){\makebox(0,0)[cc]{$\bullet$}}
\put(22,22){\makebox(0,0)[cc]{$\bullet$}}\put(15,15){\makebox(0,0)[cc]{$\bullet$}}
\put(1,2){\makebox(0,0)[cc]{${}_A$}}\put(8,9){\makebox(0,0)[cc]{${}_B$}}
\put(23,9){\makebox(0,0)[cc]{${}_A$}}\put(16,2){\makebox(0,0)[cc]{${}_C$}}
\put(1,17){\makebox(0,0)[cc]{${}_D$}}\put(6,24){\makebox(0,0)[cc]{${}_E$}}
\put(23,24){\makebox(0,0)[cc]{${}_D$}}\put(16,17){\makebox(0,0)[cc]{${}_F$}}
\put(30,0){\makebox(0,0)[cc]{$\bullet$}}\put(37,7){\makebox(0,0)[cc]{$\bullet$}}
\put(37,22){\makebox(0,0)[cc]{$\bullet$}}\put(30,15){\makebox(0,0)[cc]{$\bullet$}}
\put(31,2){\makebox(0,0)[cc]{${}_B$}}\put(38,9){\makebox(0,0)[cc]{${}_C$}}
\put(38,24){\makebox(0,0)[cc]{${}_F$}}\put(31,17){\makebox(0,0)[cc]{${}_E$}}
\end{picture}
\]
\[
\textnormal{(C)}\ 
\begin{picture}(28,28)(0,0)
\put(0,0){\line(0,1){15}}\put(0,0){\line(1,0){15}}
\put(15,0){\line(0,1){15}}\put(0,15){\line(1,0){15}}
\put(0,15){\line(1,1)7}\put(15,15){\line(1,1)7}
\put(7,22){\line(1,0){15}}\put(22,7){\line(0,1){15}}
\put(0,15){\line(1,1)7}\put(15,15){\line(1,1)7}\put(15,0){\line(1,1)7}
\multiput(15,15)(-0.1,0.0875){80}{\line(0,-1){.05}}
\multiput(15,0)(-0.5,0.4375){16}{\line(0,-1){.05}}
\multiput(0,0)(0.4375,0.4375){16}{\line(0,-1){.05}}
\multiput(7,7)(0,0.5){30}{\line(0,1){.05}}
\multiput(7,7)(0.5,0){30}{\line(1,0){.05}}
\put(0,0){\makebox(0,0)[cc]{$\bullet$}}\put(7,7){\makebox(0,0)[cc]{$\bullet$}}
\put(22,7){\makebox(0,0)[cc]{$\bullet$}}\put(15,0){\makebox(0,0)[cc]{$\bullet$}}
\put(0,15){\makebox(0,0)[cc]{$\bullet$}}\put(7,22){\makebox(0,0)[cc]{$\bullet$}}
\put(22,22){\makebox(0,0)[cc]{$\bullet$}}\put(15,15){\makebox(0,0)[cc]{$\bullet$}}
\put(1,2){\makebox(0,0)[cc]{${}_D$}}\put(8,9){\makebox(0,0)[cc]{${}_E$}}
\put(23,9){\makebox(0,0)[cc]{${}_D$}}\put(16,2){\makebox(0,0)[cc]{${}_F$}}
\put(1,17){\makebox(0,0)[cc]{${}_A$}}\put(6,24){\makebox(0,0)[cc]{${}_B$}}
\put(23,24){\makebox(0,0)[cc]{${}_A$}}\put(16,17){\makebox(0,0)[cc]{${}_C$}}
\end{picture}
\qquad\textnormal{(D)}\ 
\begin{picture}(43,28)(0,0)
\put(0,0){\line(0,1){15}}\put(0,0){\line(1,0){15}}
\put(15,0){\line(0,1){15}}\put(0,15){\line(1,0){15}}
\put(0,15){\line(1,1)7}\put(15,15){\line(1,1)7}
\put(7,22){\line(1,0){15}}
\put(0,15){\line(1,1)7}\put(15,15){\line(1,1)7}\put(30,0){\line(1,1)7}
\put(15,0){\line(1,0){15}}
\put(30,0){\line(0,1){15}}\put(15,15){\line(1,0){15}}
\put(15,15){\line(1,1)7}\put(30,15){\line(1,1)7}
\put(22,22){\line(1,0){15}}\put(37,7){\line(0,1){15}}
\put(0,15){\line(1,1)7}\put(15,15){\line(1,1)7}
\multiput(15,15)(-0.1,0.0875){80}{\line(0,-1){.05}}
\multiput(30,15)(-0.1,0.0875){80}{\line(0,-1){.05}}
\multiput(15,0)(-0.5,0.4375){16}{\line(0,-1){.05}}
\multiput(0,0)(0.4375,0.4375){16}{\line(0,-1){.05}}
\multiput(7,7)(0,0.5){30}{\line(0,1){.05}}
\multiput(7,7)(0.5,0){30}{\line(1,0){.05}}
\multiput(30,0)(-0.5,0.4375){16}{\line(0,-1){.05}}
\multiput(15,0)(0.4375,0.4375){16}{\line(0,-1){.05}}
\multiput(22,7)(0,0.5){30}{\line(0,1){.05}}
\multiput(22,7)(0.5,0){30}{\line(1,0){.05}}
\put(0,0){\makebox(0,0)[cc]{$\bullet$}}\put(7,7){\makebox(0,0)[cc]{$\bullet$}}
\put(22,7){\makebox(0,0)[cc]{$\bullet$}}\put(15,0){\makebox(0,0)[cc]{$\bullet$}}
\put(0,15){\makebox(0,0)[cc]{$\bullet$}}\put(7,22){\makebox(0,0)[cc]{$\bullet$}}
\put(22,22){\makebox(0,0)[cc]{$\bullet$}}\put(15,15){\makebox(0,0)[cc]{$\bullet$}}
\put(1,2){\makebox(0,0)[cc]{${}_C$}}\put(8,9){\makebox(0,0)[cc]{${}_A$}}
\put(23,9){\makebox(0,0)[cc]{${}_C$}}\put(16,2){\makebox(0,0)[cc]{${}_B$}}
\put(1,17){\makebox(0,0)[cc]{${}_F$}}\put(6,24){\makebox(0,0)[cc]{${}_D$}}
\put(23,24){\makebox(0,0)[cc]{${}_F$}}\put(16,17){\makebox(0,0)[cc]{${}_E$}}
\put(30,0){\makebox(0,0)[cc]{$\bullet$}}\put(37,7){\makebox(0,0)[cc]{$\bullet$}}
\put(37,22){\makebox(0,0)[cc]{$\bullet$}}\put(30,15){\makebox(0,0)[cc]{$\bullet$}}
\put(31,2){\makebox(0,0)[cc]{${}_A$}}\put(38,9){\makebox(0,0)[cc]{${}_B$}}
\put(38,24){\makebox(0,0)[cc]{${}_E$}}\put(31,17){\makebox(0,0)[cc]{${}_D$}}
\end{picture}
\] \caption{Maximal orders containing a GEO.}
\end{figure}
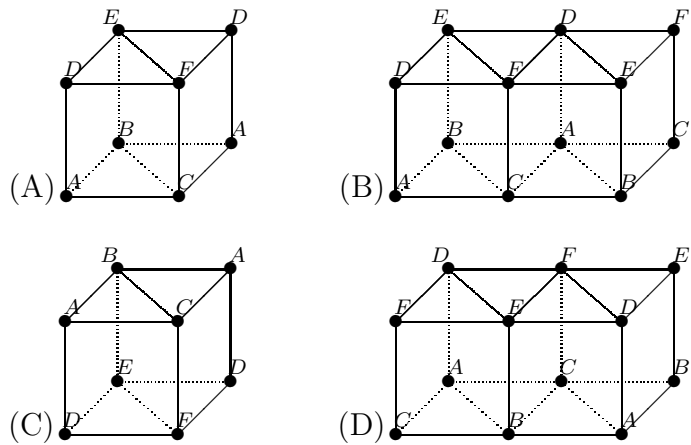
\end{example}

\end{document}